\documentclass[11pt]{article}
\usepackage{amsmath, amsfonts, amsthm, amssymb,verbatim}
\usepackage{graphicx}
        \newtheorem{theorem}{Theorem}[section]
\newtheorem{definition}[theorem]{Definition}
        
        \newtheorem{lemma}[theorem]{Lemma}
        \newtheorem{corollary}[theorem]{Corollary}
        \newtheorem{remark}[theorem]{Remark}

\numberwithin{equation}{section}

\newcommand \R      {\mathbb R}
\newcommand \be     {\begin{equation}}
\newcommand \ee     {\end{equation}}
\usepackage{makeidx}
\usepackage{showidx}
\newcommand{\T}{{\mathbb{T}}}

\newcommand{\N}{{\mathbb{N}}}

\newcommand{\cP}{{\mathcal{P}}}

\newcommand{\veps}{\varepsilon}

\newcommand{\sig}{^{\sigma}}

\newcommand{\kap}{^{\kappa}}

\newcommand{\tx}[1]{\quad\text{#1}\quad}

\newcommand{\neu}{\\ \\}

\newcommand{\tee}{t_1,t_2,\ldots,t_n}
\newcommand{\deltee}{\Delta_1t_1\Delta_2t_2\ldots\Delta_nt_n}

\begin{document}

\title{Two Dimensional Integral  Inequalities on Time Scales}
\author{ Svetlin G. Georgiev\footnote{Sorbonne University, Paris, France. E-mail:svetlingeorgiev1@gmail.com}, Goverdan Khadekar\footnote{RTM Nagpur University, Faculty of Science, Department of Mathematics, India. E-mail:gkhadekar@yahoo.com} and Praveen Kumar\footnote{RTM Nagpur University, Faculty of Science, Department of Mathematics, India. E-mail:pkumar6743@gmail.com} }
\maketitle
\begin{abstract}
In this paper we formulate and prove Wendroff's inequalities on time scales. Next,  we deduct some Pachpatte's inequalities.

\end{abstract}

\maketitle
\section{Introduction}
The theory of time scales was initiated by Hilger \cite{hilger} in his Ph.D. thesis in 1988 in order to contain both difference and differential calculus in a consistent way. Since then many authors have investigated  various aspects of the theory of dynamic equations on time scales. For example, the monographes \cite{bohner1, georgiev} and the references cited therein. At the same time, in the  papers \cite{ag}, \cite{and2}, \cite{10}, \cite{14}, \cite{erbe}, \cite{12}, \cite{li}, \cite{ben}, \cite{15}, \cite{17}, \cite{xing}, \cite{xu}, \cite{13} and references therein have studied the theory of integral inequalities on time scales.  In \cite{and1} and \cite{and3} the author establishes some general nonlinear dynamic inequalities on general time scales involving functions in two independent variables and the author extends double sum  and integral inequalities of Hilbert-Pachpatte type to general dynamic double integral inequalities on time scales. In \cite{and2} and \cite{fer} are established some Wendroff's type inequalities, and in \cite{and2} they are established some Wendroff's type inequalities by Picard operators. In this paper we study some two-dimensional integral and integro-dynamic    Pachpatte's  inequalities on time scales.

\noindent The paper is organized as follows. In the next Section we give some basic definitions and facts of the time scale calculus.  In Section 4 we get some integral and integro-dynamic  Pachpatte's inequalities on time scales.
\section{Time Scales Essentials}
This section is devoted to a brief exposition of the time scale calculus. A detailed discussion of the time scale calculus is beyond the scope of this book, for this reason the author confine to outlining a minimal set of properties needed in the further proceeding. The presentation in this section follows the books \cite{bohner1} and \cite{georgiev}.
\begin{definition}
\vbox{\index{Time Scale}} A time scale  is an arbitrary nonempty closed subset of the real numbers.
\end{definition}
\noindent
We will denote a time scale by the symbol $\mathbb{T}$.
\begin{definition}
\vbox{\index{Forward Jump Operator}} For $t\in \mathbb{T}$ we define the  forward jump operator $\displaystyle{\sigma:\mathbb{T}\longmapsto \mathbb{T}}$ as follows
$$
\displaystyle{\sigma(t)=\inf\{s\in \mathbb{T}:s>t\}.}
$$

\noindent We note that $\displaystyle{\sigma(t)\geq t}$ for any $\displaystyle{t\in \mathbb{T}}$. If $\sigma(t)>t$, then we say that $t$ is right-scattered. If $\sigma(t)=t$ and $t<\sup\mathbb{T}$, then we say that $t$ is right-dense.
\end{definition}
\begin{definition}
\vbox{\index{Backward Jump Operator}} For $t\in \mathbb{T}$ we define the  backward jump operator $\displaystyle{\rho:\mathbb{T}\longmapsto \mathbb{T}}$ as follows
$$
\displaystyle{\rho(t)=\sup\{s\in \mathbb{T}:s<t\}.}
$$

\noindent We note that $\displaystyle{\rho(t)\leq t}$ for any $\displaystyle{t\in \mathbb{T}}$.
If $\rho(t)<t$, then we say that $t$ is left-scattered. If $\rho(t)=t$ and $t>\inf\mathbb{T}$, then we say that $t$ is left-dense.
\end{definition}
\begin{definition}
We set
$$
\displaystyle{\inf \O=\sup\mathbb{T},\quad \sup\O=\inf\mathbb{T}.}
$$
\end{definition}

\noindent Let $\displaystyle{\mathbb{T}}$ be a time scale with forward jump operator and  backward jump operator  $\displaystyle{\sigma}$ and  $\displaystyle{\rho}$, respectively.
\begin{definition}
We define the set
$$
\mathbb{T}^{\kappa}=\left\{
\begin{array}{l}
\mathbb{T}\backslash (\rho(\sup\mathbb{T}), \sup\mathbb{T}]\quad {\rm if}\quad \sup\mathbb{T}<\infty\\ \\
\mathbb{T}\quad {\rm otherwise}.
\end{array}
\right.
$$
\end{definition}

\begin{definition}
The  graininess function $\displaystyle{\mu:\mathbb{T}\longmapsto [0, \infty)}$  is defined by
$$
\displaystyle{\mu(t)=\sigma(t)-t.}
$$\vbox{\index{Graininess Function}}
\end{definition}

\begin{definition}\label{appendix21}
Assume that $\displaystyle{f:\mathbb{T}\longmapsto \R}$ is a function and let $\displaystyle{t\in \mathbb{T}^{\kappa}}$. We define $\displaystyle{f^{\Delta}(t)}$ to be the number, provided it exists, as follows: for any $\epsilon>0$ there is a neighbourhood $\displaystyle{U}$ of $t$, $\displaystyle{U=(t-\delta, t+\delta)\cap\mathbb{T}}$ for some $\delta>0$, such that
$$
\displaystyle{|f(\sigma(t))-f(s)-f^{\Delta}(t)(\sigma(t)-s)|\leq \epsilon |\sigma(t)-s|\quad {\rm for}\quad {\rm all}\quad s\in U,\quad s\ne \sigma(t).}
$$
We say $\displaystyle{f^{\Delta}(t)}$ the  delta or  Hilger derivative of $f$ at $t$. \vbox{\index{Delta Derivative}} \vbox{\index{Hilger Derivative}}

\noindent We say that $f$ is  delta or Hilger differentiable, shortly  differentiable, \vbox{\index{Delta Differentiable Function}} \vbox{\index{Hilger Differentiable Function}} in $\displaystyle{T^{\kappa}}$ if $\displaystyle{f^{\Delta}(t)}$ exists for all $\displaystyle{t\in \mathbb{T}^{\kappa}}$. The function $\displaystyle{f^{\Delta}:\mathbb{T}\longmapsto\R}$ is said to be  delta derivative or Hilger derivative, shortly derivative, of $f$ in $\displaystyle{T^{\kappa}}$.
\end{definition}
\begin{remark}
If $\mathbb{T}=\R$, then the delta derivative coincides with the classical derivative.
\end{remark}
\noindent Note that the delta derivative is well-defined. For the properties of the delta derivative we refer the reader to \cite{bohner1} and \cite{georgiev}.

\begin{definition}
A function $\displaystyle{f:\mathbb{T}\longmapsto\R}$ is called regulated \vbox{\index{Regulated Function}} provided its right-sided limits  exist(finite) at all right-dense points in $\displaystyle{\mathbb{T}}$ and its left-sided limits exist(finite) at all left-dense points in $\displaystyle{\mathbb{T}}$.
\end{definition}
\begin{definition}
A continuous function $\displaystyle{f:\mathbb{T}\longmapsto\R}$ is called pre-differentiable with region of differentiation $\displaystyle{D}$, provided \vbox{\index{Pre-Differentiable Function}}
\begin{enumerate}
\item $\displaystyle{D\subset \mathbb{T}^{\kappa}}$,
\item $\displaystyle{\mathbb{T}^{\kappa}\backslash D}$ is countable and contains no right-scattered elements of $\displaystyle{\mathbb{T}}$,
\item $\displaystyle{f}$ is differentiable at each $\displaystyle{t\in D}$.
\end{enumerate}
\end{definition}
\begin{theorem}[\cite{bohner1}, \cite{georgiev}]\label{theorem11123} Let $\displaystyle{t_0\in\mathbb{T}}$, $\displaystyle{x_0\in\R}$, $\displaystyle{f:\mathbb{T}^{\kappa}\longmapsto\R}$ be given regulated map. Then there exists exactly one pre-differentiable function $\displaystyle{F}$ satisfying
$$
\displaystyle{F^{\Delta}(t)=f(t)\quad {\rm for}\quad {\rm all}\quad t\in D,\quad F(t_0)=x_0.}
$$
\end{theorem}
\begin{definition}
Assume $\displaystyle{f:\mathbb{T}\longmapsto \R}$ is a regulated function. Any function $\displaystyle{F}$ by Theorem \ref{theorem11123} is called a pre-antiderivative of $\displaystyle{f}$. \vbox{\index{Pre-Antiderivative}}
 We define the indefinite integral of a regulated function $\displaystyle{f}$  by
 $$
 \displaystyle{\int f(t)\Delta t=F(t)+c,}
 $$
 where $\displaystyle{c}$ is an arbitrary constant and $\displaystyle{F}$ is a pre-antiderivative of $\displaystyle{f}$. We define the Cauchy integral by
 $$
 \displaystyle{\int_{\tau}^s f(t)\Delta t=F(s)-F(\tau)\quad {\rm for}\quad {\rm all}\quad \tau, s\in\mathbb{T}.}
 $$
 A function $\displaystyle{F:\mathbb{T}\longmapsto\R}$ is called an antiderivative of $\displaystyle{f:\mathbb{T}\longmapsto \R}$ provided
 $$
 \displaystyle{F^{\Delta}(t)=f(t)\quad {\rm holds}\quad {\rm for}\quad {\rm all}\quad t\in \mathbb{T}^{\kappa}.}
 $$
 \end{definition}
For properties of the delta integral we refer the reader to \cite{bohner1} and \cite{georgiev}.
\begin{definition}
We say that $f: \mathbb{T}\to \mathbb{R}$ is rd-continuous provided $f$ is continuous at each right-dense point of $\mathbb{T}$ and has a finite left-dense limit at each left-dense point of $\mathbb{T}$. The set of rd-continuous functions will be denoted by $\mathcal{C}_{rd}(\mathbb{T})$ and the set of functions that are differentiable and whose derivative is rd-continuous is denoted by $\mathcal{C}_{rd}^1(\mathbb{T})$.
\end{definition}
\begin{definition}
We say that $f: \mathbb{T}\to \R$  is regressive provided
\begin{equation*}
1+\mu(t) f(t)\ne 0,\quad t\in \mathbb{T}.
\end{equation*}
We denote by $\mathcal{R}$ the set of all regressive and rd-continuous functions. Define
\begin{equation*}
\mathcal{R}_+=\left\{ f\in \mathcal{R}: 1+\mu(t) f(t)>0,\quad t\in \mathbb{T}\right\}.
\end{equation*}
\end{definition}
\begin{definition}
If $f, g\in \mathcal{R}$, then we define
\begin{equation*}
f\oplus g=f+g+\mu fg,\quad \ominus g=-\frac{g}{1+\mu g},\quad f\ominus g= f\oplus(\ominus g).
\end{equation*}
\end{definition}
\begin{definition}
If $f: \mathbb{T}\to \R$ is rd-continuous and regressive, then the exponential function $e_p(\cdot, t_0)$ is for each fixed $t_0\in \mathbb{T}$ the unique solution of the initial value problem
\begin{equation*}
x^{\Delta}=f(t) x,\quad x(t_0)=1\quad \textrm{on}\quad \mathbb{T}.
\end{equation*}
\end{definition}
For properties of regressive functions, rd-continuous functions and the exponential function we refer the reader to \cite{bohner1} and \cite{georgiev}.
\begin{lemma}[Comparison Lemma]
\label{lemma2.1} Let $x\in \mathcal{C}_{rd}^1(\mathbb{T})$, $f, g\in \mathcal{C}_{rd}(\mathbb{T})$, $g\in \mathcal{R}^+$, $a\in \mathbb{T}$ and
\begin{equation*}
x^{\Delta}(t)\leq f(t)+g(t) x(t),\quad t\geq a.
\end{equation*}
Then
\begin{equation*}
x(t)\leq x(a) e_g(t, a)+\int_a^t f(s) e_{\ominus g}(\sigma(s), t)\Delta s,\quad t\geq a.
\end{equation*}
\end{lemma}
Let $n\in\N$ be fixed. For each $i\in\{1,2,\ldots,n\}$,
we denote by $\T_i$ a time scale.

\begin{definition}
The set\vbox{\index{ lambda@$\Lambda^n$}}
$$ \Lambda^n=\T_1\times\T_2\times\cdots\times\T_n
   =\left\{t=(\tee):\;t_i\in\T_i,\;i=1,2,\ldots,n\right\} $$
is called
an {\em $n$-dimensional time scale}.\vbox{\index{Higher-Dimensional time scale}}
\end{definition}

\begin{definition}
Let $\sigma_i$, $i\in\{1,2,\ldots,n\}$,
be the forward jump operator in $\T_i$.
The operator $\sigma:\Lambda^n\to\Lambda^n$ defined by
$$ \sigma(t)=(\sigma_1(t),\sigma_2(t),\ldots,\sigma_n(t)) $$
is said to be
the {\em forward jump operator}\vbox{\index{Forward Jump Operator}}
in $\Lambda^n$.
\end{definition}

\begin{definition}
Let $\rho_i$, $i\in\{1,2,\ldots,n\}$,
be the backward jump operator in $\T_i$.
The operator $\rho:\Lambda^n\to\R^n$ defined by
$$ \rho(t)=(\rho_1(t_1),\rho_2(t_2),\ldots,\rho_n(t_n)),\quad
   t=(\tee)\in\Lambda^n, $$
is said to be
the {\em backward jump operator}\vbox{\index{Backward jump operator}}
in $\Lambda^n$.
\end{definition}

\begin{definition}
For $x=(x_1,x_2,\ldots,x_n)\in\R^n$ and $y=(y_1,y_2,\ldots,y_n)\in\R^n$,
we write
$$ x\geq y $$
whenever
$$ x_i\geq y_i \tx{for all} i=1,2,\ldots,n. $$
In a similar way, we understand $x>y$ and $x<y$ and $x\leq y$.
\end{definition}
\begin{definition}
The {\em graininess}\vbox{\index{Graininess}}
function $\mu:\Lambda^n\to[0,\infty)^n$ is defined by
$$ \mu(t)=(\mu_1(t_1),\mu_2(t_2),\ldots,\mu_n(t_n)),\quad
   t=(\tee)\in\Lambda^n. $$
\end{definition}

\begin{definition}
Let $f:\Lambda\to\R$.
We introduce the following notations
\begin{eqnarray*}
   f\sig(t)
   &=& f\left(\sigma_1(t_1),\sigma_2(t_2),\ldots,\sigma_n(t_n)\right),\neu
   f_i^{\sigma_i}(t)
   &=& f\left(t_1,\ldots,t_{i-1},\sigma_i(t_i),t_{i+1},\ldots,t_n\right),\neu
   f_{i_1i_2\ldots i_l}^{\sigma_{i_1}\sigma_{i_2}\ldots \sigma_{i_l}}(t)
   &=& f\left(\ldots,\sigma_{i_1}(t_{i_1}),\ldots,
              \sigma_{i_2}(t_{i_2}),\ldots, \sigma_{i_l}(t_{i_l}),\ldots\right),
\end{eqnarray*}
where $1\leq i_1<i_2<\ldots<i_l\leq n$, $i_m\in\N$, $m\in\{1,2,\ldots,l\}$,
$l\in\N$.
\end{definition}
\begin{definition}
We set
\begin{gather*}
   \Lambda^{\kappa n}=\T_1\kap\times\T_2\kap\times\ldots\times\T_n\kap,\neu
   \Lambda_i^{\kappa_in}=\T_1\times\ldots\times\T_{i-1}\times
   \T_i\kap\times\T_{i+1}\times\ldots\times\T_n,\quad i=1,2,\ldots,n,\neu
   \Lambda_{i_1i_2\ldots i_l}^{\kappa_{i_1}\kappa_{i_2}\ldots \kappa_{i_l}n}
   =\ldots\times\T_{i_1}\kap\times\ldots\times\T_{i_2}\kap\times\ldots\times
   \T_{i_l}\kap\times \ldots,
\end{gather*}
where $1\leq i_1<i_2<\ldots<i_l\leq n$, $i_m\in\N$, $m=1,2,\ldots,l$.
\end{definition}

\begin{remark}
If $(i_1,i_2,\ldots,i_l)=(1,2,\ldots,n)$, then
$$ \Lambda_{i_1i_2\ldots i_l}^{\kappa_1\kappa_2\ldots \kappa_ln}
   =\Lambda^{\kappa n}. $$
\end{remark}

\begin{definition}
Assume that $f:\Lambda^n\to\R$ is a function and let
$t\in\Lambda_{i}^{\kappa_i n}$.
We define
$$ \frac{\partial f(\tee)}{\Delta_it_i}
   =\frac{\partial f(t)}{\Delta_it_i}
   =\frac{\partial f}{\Delta_it_i}(t)
   =f_{t_i}^{\Delta_i}(t) $$
to be the number, provided it exists, with the property that
for any $\veps_i>0$, there exists a neighbourhood
$$ U_i=(t_i-\delta_i,t_i+\delta_i)\cap\T_i, $$
for some $\delta_i>0$, such that
\begin{multline}\label{13}
   \Bigl|f(t_1,\ldots,t_{i-1},\sigma_i(t_i),t_{i+1},\ldots,t_n)
   -f(t_1,\ldots,t_{i-1},s_i,t_{i+1},\ldots,t_n)\neu
   -f_{t_i}^{\Delta_i}(t)(\sigma_i(t_i)-s_i)\Bigr|
   \leq\veps_i |\sigma_i(t_i)-s_i| \tx{for all} s_i\in U_i.
\end{multline}
We call $f_{t_i}^{\Delta_i}(t)$ the {\em partial delta derivative}
(or {\em partial Hilger derivative}) of $f$ with respect to $t_i$ at $t$.
We say that $f$ is {\em partial delta differentiable}
(or {\em partial Hilger differentiable}) with respect to
$t_i$ in $\Lambda_i^{\kappa_i n}$
if $f^{\Delta_i}_{t_i}(t)$ exists for all $t\in\Lambda_i^{\kappa_i n}$.
The function $f_{t_i}^{\Delta_i}:\Lambda_i^{\kappa_i n}\to\R$ is said
to be the {\em partial delta derivative}\vbox{\index{Partial Delta Derivative}}
(or {\em partial Hilger derivative})\vbox{\index{Partial Hilger Derivative}}
with respect to $t_i$ of $f$ in $\Lambda_i^{\kappa_i n}$.
\end{definition}

The partial delta derivative is well defined. For the properties of the partial delta derivative we refer the reader to \cite{georgiev}.

\begin{definition}
For a function $f:\Lambda^n\to\R$,
we shall talk about the {\em second-order partial delta derivative}
with\vbox{\index{Partial Delta Derivative!second-order}}
respect to $t_i$ and $t_j$, $i,j\in\{1,2,\ldots,n\}$,
$f_{t_i t_j}^{\Delta_i\Delta_j}$,
provided $f_{t_i}^{\Delta_i}$ is partial delta differentiable
with respect to $t_j$ on
$\Lambda_{ij}^{\kappa_i\kappa_jn}
 =\left(\Lambda_i^{\kappa_i n}\right)_j^{\kappa_j n}$
with partial delta derivative
$$ f_{t_it_j}^{\Delta_i\Delta_j}
   =\left(f_{t_i}^{\Delta_i}\right)_{t_j}^{\Delta_j}:
   \Lambda_{ij}^{\kappa_i\kappa_jn}\to\R. $$
For $i=j$, we will write
$$ f_{t_it_i}^{\Delta_i\Delta_i}=f_{t_i^2}^{\Delta_i^2}. $$
Similarly, we\vbox{\index{Partial Delta Derivative!higher-order}}
define {\em higher-order partial delta derivatives}
$$ f_{t_it_j\dots t_l}^{\Delta_i\Delta_j\ldots\Delta_l}:
   \Lambda_{ij\ldots l}^{\kappa_i\kappa_j\ldots\kappa_l n}\to\R. $$
For $t\in\Lambda^n$, we define
$$ \sigma^2(t)=\sigma(\sigma(t))
   =(\sigma_1(\sigma_1(t_1)),\sigma_2(\sigma_2(t_2)),\ldots,
              \sigma_n(\sigma_n(t_n))) $$.
\end{definition}

Now we will introduce the conception for multiple integration on time scales.
Suppose $a_i<b_i$ are points in $\T_i$ and
$[a_i,b_i)$ is the half-closed bounded interval in $\T_i$, $i\in\{1,\ldots,n\}$.
Let us introduce a ``rectangle'' in
$\Lambda^n=\T_1\times\T_2\times\ldots\times\T_n$ by
\begin{eqnarray*}
   R
   &=& [a_1,b_1)\times[a_2,b_2)\times\ldots[a_n,b_n)\neu
   &=& \left\{(\tee):\; t_i\in[a_i,b_i),\; i=1,2,\ldots,n\right\}.
\end{eqnarray*}
Let
$$ a_i=t_i^0<t_i^1<\ldots<t_i^{k_i}=b_i. $$

\begin{definition}\vbox{\index{Partition of Interval}}
We call the collection of intervals
$$ P_i=\left\{[t_i^{j_i-1},t_i^{j_i}):\; j_i=1,\ldots,k_i\right\},\quad
   i=1,2,\ldots,n, $$
a {\em $\Delta_i$-partition} of $[a_i,b_i)$ and denote the set of
all $\Delta_i$-partitions of $[a_i,b_i)$ by $P_i([a_i,b_i))$.
\end{definition}

\begin{definition}
Let
\begin{equation}\label{60}
\begin{array}{c}
   R_{j_1 j_2\ldots j_n}
   =[t_1^{j_1-1},t_1^{j_1})\times[t_2^{j_2-1},t_2^{j_2})\times\ldots
    \times[t_n^{j_n-1},t_n^{j_n})\neu
   1\leq j_i\leq k_i,\quad i=1,2,\ldots,n.
\end{array}
\end{equation}
We call the collection
\begin{equation}\label{61}
   P=\left\{R_{j_1j_2\ldots j_n}:\; 1\leq j_i\leq k_i,\; i=1,2,\ldots,n\right\}
\end{equation}
a {\em $\Delta$-partition}\vbox{\index{Partition of Rectangle}}
of $R$,
generated by the $\Delta_i$-partitions $P_i$ of $[a_i,b_i)$, and we write
$$ P=P_1\times P_2\times\ldots\times P_n. $$
The set of all $\Delta$-partitions of $R$ is denoted by $\cP(R)$.
Moreover, for a bounded function $f:R\to\R$, we set
\begin{eqnarray*}
   M
   &=& \sup\{f(\tee):\; (\tee)\in R\},\neu
   m
   &=& \inf\{f(\tee):\; (\tee)\in R\},\neu
   M_{j_1j_2\ldots j_n}
   &=& \sup\{f(\tee):\;
       (\tee)\in R_{j_1j_2\ldots j_n}\},\neu
   m_{j_1j_2\ldots j_n}
   &=& \inf\{f(\tee):\;
       (\tee)\in R_{j_1j_2\ldots j_n}\}.
\end{eqnarray*}
\end{definition}

\begin{definition}
The {\em upper Darboux $\Delta$-sum}\vbox{\index{Darboux $\Delta$-sum!upper}}
$U(f,P)$ and
the {\em lower Darboux $\Delta$-sum}\vbox{\index{Darboux $\Delta$-sum!lower}}
$L(f,P)$
with respect to $P$ are defined by
$$ U(f,P)=\sum_{j_1=1}^{k_1}\sum_{j_2=1}^{k_2}\ldots\sum_{j_n=1}^{k_n}
   M_{j_1j_2\ldots j_n}(t_1^{j_1}-t_1^{j_1-1})(t_2^{j_2}-t_2^{j_2-1})\ldots
   (t_n^{j_n}-t_n^{j_n-1}) $$
and
$$ L(f,P)=\sum_{j_1=1}^{k_1}\sum_{j_2=1}^{k_2}\ldots\sum_{j_n=1}^{k_n}
   m_{j_1j_2\ldots j_n}(t_1^{j_1}-t_1^{j_1-1})(t_2^{j_2}-t_2^{j_2-1})\ldots
   (t_n^{j_n}-t_n^{j_n-1}). $$
\end{definition}

\begin{definition}
The {\em upper Darboux $\Delta$-integral}\vbox{\index{Darboux $\Delta$-integral!upper}}
$U(f)$ of $f$ over $R$ and
the {\em lower Darboux $\Delta$-integral}\vbox{\index{Darboux $\Delta$-integral!lower}}
$L(f)$ of $f$ over $R$ are defined by
$$ U(f)=\inf\{U(f,P):P\in\cP(R)\} \tx {and} L(f)=\sup\{L(f,P):P\in\cP(R)\}. $$
\end{definition}

We have  that $U(f)$ and $L(f)$ are finite real numbers.

\begin{definition}\label{definition37}
We say that $f$ is {\em $\Delta$-integrable} over $R$ provided $L(f)=U(f)$.
In this case, we write
$$ \int_Rf(\tee)\deltee $$
for this common value.
We call this integral
the {\em Darboux $\Delta$-integral}.\vbox{\index{Darboux $\Delta$-integral}}
\end{definition}

\section{Pachpatte's Inequalities}
Let $\mathbb{T}_1$ and $\mathbb{T}_2$ be time scales with forward jump operators and delta differentiation operators $\sigma_1$, $\sigma_2$ and $\Delta_1$, $\Delta_2$ respectively, which contain positive numbers and $0\in \mathbb{T}_1$, $0\in \mathbb{T}_2$.
\begin{theorem}[Pachpatte's Inequality]
\label{theorem6.4}  Let $u, p, q \in \mathcal{C}((\mathbb{R}_+\bigcap \mathbb{T}_1)\times (\mathbb{R}_+\bigcap \mathbb{T}_2))$ be nonnegative functions,
$$k(\cdot, \cdot, s_1, s_2)\in \mathcal{C}^2\left((\mathbb{R}_+\bigcap \mathbb{T}_1)\times (\mathbb{R}_+\bigcap \mathbb{T}_2)\right),\quad (s_1, s_2)\in (\mathbb{R}_+\bigcap \mathbb{T}_1)\times (\mathbb{R}_+\bigcap \mathbb{T}_2),$$ and its partial derivatives
\begin{eqnarray*}
k_{t_1}^{\Delta_1}(t_1, t_2, s_1, s_2), &k_{t_1}^{\Delta_1}(\sigma_1(t_1), t_2, s_1, s_2),&\\ \\
k_{t_2}^{\Delta_2}(t_1, t_2, s_1, s_2), & k_{t_2}^{\Delta_2}(t_1, \sigma_2(t_2), s_1, s_2),&\\ \\
k_{t_1t_2}^{\Delta_1\Delta_2}(t_1, t_2, s_1, s_2), &(t_1, t_2), (s_1, s_2)\in (\mathbb{R}_+\bigcap \mathbb{T}_1)\times (\mathbb{R}_+\bigcap \mathbb{T}_2)&,
\end{eqnarray*}
be nonnegative  functions. If
\begin{equation*}
u(t_1, t_2) \leq p(t_1, t_2) +q(t_1, t_2) \int_0^{t_1}\int_0^{t_2} k(t_1, t_2, s_1, s_2) u(s_1, s_2)\Delta_2s_2\Delta_1s_1,
\end{equation*}
$(t_1, t_2)\in (\mathbb{R}_+\bigcap \mathbb{T}_1)\times (\mathbb{R}_+\bigcap \mathbb{T}_2)$, then
\begin{equation*}
u(t_1, t_2) \leq p(t_1, t_2)+q(t_1, t_2) A(t_1, t_2) e_c(t_2, 0),
\end{equation*}
$(t_1, t_2)\in (\mathbb{R}_+\bigcap \mathbb{T}_1)\times (\mathbb{R}_+\bigcap \mathbb{T}_2)$, where
\begin{eqnarray*}
a(t_1, t_2)&=& k(\sigma_1(t_1), \sigma_2(t_2), t_1, t_2) p(t_1, t_2)\\ \\
&&+\int_0^{t_2} k_{t_2}^{\Delta_2}(\sigma_1(t_1), t_2, t_1, s_2) p(t_1, s_2) \Delta_2 s_2,\\ \\
&&+\int_0^{t_1} k_{t_1}^{\Delta_1}(t_1, \sigma_2(t_2), s_1, t_2) p(s_1, t_2) \Delta_1 s_1\\ \\
&&+\int_0^{t_1}\int_0^{t_2} k_{t_1t_2}^{\Delta_1\Delta_2}(t_1, t_2, s_1, s_2) p(s_1, s_2) \Delta_2 s_2 \Delta_1 s_1,\\ \\
b(t_1, t_2)&=& k(\sigma_1(t_1), \sigma_2(t_2), t_1, t_2) q(t_1, t_2)\\ \\
&&+\int_0^{t_2} k_{t_2}^{\Delta_2}(\sigma_1(t_1), t_2, t_1, s_2) q(t_1, s_2) \Delta_2 s_2\\ \\
&&+\int_0^{t_1} k_{t_1}^{\Delta_1}(t_1, \sigma_2(t_2), s_1, t_2) q(s_1, t_2) \Delta_1 s_1\\ \\
&&+\int_0^{t_1}\int_0^{t_2} k_{t_1t_2}^{\Delta_1\Delta_2}(t_1, t_2, s_1, s_2) q(s_1, s_2) \Delta_2 s_2 \Delta_1 s_1,\\ \\
A(t_1, t_2)&=& \int_0^{t_1}\int_0^{t_2} a(s_1, s_2) \Delta_2 s_2 \Delta_1 s_1,\\ \\
c(t_1, t_2)&=& \int_0^{t_2} b(t_1, s_2) \Delta_2 s_2,
\end{eqnarray*}
$(t_1, t_2)\in (\mathbb{R}_+\bigcap \mathbb{T}_1)\times (\mathbb{R}_+\bigcap \mathbb{T}_2)$.
\end{theorem}
\begin{proof}
Let
\begin{equation*}
z(t_1, t_2)= \int_0^{t_1}\int_0^{t_2} k(t_1, t_2, s_1, s_2) u(s_1, s_2) \Delta_2 s_2\Delta_1 s_1,\quad (t_1, t_2)\in (\mathbb{R}_+\bigcap \mathbb{T}_1)\times (\mathbb{R}_+\bigcap \mathbb{T}_2).
\end{equation*}
Then
\begin{equation*}
u(t_1, t_2)\leq p(t_1, t_2)+q(t_1, t_2) z(t_1, t_2),\quad (t_1, t_2)\in (\mathbb{R}_+\bigcap \mathbb{T}_1)\times (\mathbb{R}_+\bigcap \mathbb{T}_2),
\end{equation*}
and $z(t_1, t_2)$ is a nondecreasing function in each  variable $t_1$, $t_2$, $(t_1, t_2)\in (\mathbb{R}_+\bigcap \mathbb{T}_1)\times (\mathbb{R}_+\bigcap \mathbb{T}_2)$. Then
\begin{eqnarray*}
z(0, 0)&=& 0,\\ \\
z_{t_1}^{\Delta_1}(t_1, t_2)&=& \int_0^{t_2} k(\sigma_1(t_1), t_2, t_1, s_2) u(t_1, s_2)\Delta_2 s_2\\ \\
&&+ \int_0^{t_1} \int_0^{t_2} k_{t_1}^{\Delta_1}(t_1, t_2, s_1, s_2) u(s_1, s_2) \Delta_2 s_2 \Delta_1 s_1,\\ \\
z_{t_1t_2}^{\Delta_1\Delta_2}(t_1, t_2)&=& k(\sigma_1(t_1), \sigma_2(t_2), t_1, t_2) u(t_1, t_2) \\ \\
&&+ \int_0^{t_2} k_{t_2}^{\Delta_2}(\sigma_1(t_1), t_2, t_1, s_2) u(t_1, s_2)\Delta_2s_2\\ \\
&&+\int_0^{t_1} k_{t_1}^{\Delta_1}(t_1, \sigma_2(t_2), s_1, t_2) u(s_1, t_2) \Delta_1 s_1\\ \\
&&+\int_0^{t_1}\int_0^{t_2} k_{t_1t_2}^{\Delta_1\Delta_2}(t_1, t_2, s_1, s_2) u(s_1, s_2)\Delta_2 s_2 \Delta_1 s_1\\ \\
&\leq& k(\sigma_1(t_1), \sigma_2(t_2), t_1, t_2)p(t_1, t_2)\\ \\
&&+k(\sigma_1(t_1), \sigma_2(t_2), t_1, t_2) q(t_1, t_2) z(t_1, t_2)\\ \\
&&+\int_0^{t_2} k_{t_2}^{\Delta_2}(\sigma_1(t_1), t_2, t_1, s_2) p(t_1, s_2) \Delta_2 s_2\\ \\
&&+\int_0^{t_2} k_{t_2}^{\Delta_2}(\sigma_1(t_1), t_2, t_1, s_2) q(t_1, s_2)z(t_1, s_2)\Delta_2 s_2\\ \\
&&+\int_0^{t_1} k_{t_1}^{\Delta_1}(t_1, \sigma_2(t_2), s_1, t_2) p(s_1, t_2) \Delta_1 s_1
\end{eqnarray*}
\begin{eqnarray*}
&&+\int_0^{t_1} k_{t_1}^{\Delta_1}(t_1, \sigma_2(t_2), s_1, t_2)q(s_1, t_2) z(s_1, t_2) \Delta_1 s_1\\ \\
&&+ \int_0^{t_1}\int_0^{t_2} k_{t_1t_2}^{\Delta_1\Delta_2}(t_1, t_2, s_1, s_2) p(s_1, s_2) \Delta_2 s_2 \Delta_1 s_1\\ \\
&&+\int_0^{t_1}\int_0^{t_2} k_{t_1t_2}^{\Delta_1\Delta_2} (t_1, t_2, s_1, s_2) q(s_1, s_2) z(s_1, s_2) \Delta_2 s_2\Delta_1 s_1
\\ \\
&\leq& a(t_1, t_2)\\ \\
&&+ k(\sigma_1(t_1), \sigma_2(t_2), t_1, t_2) q(t_1, t_2) z(t_1, t_2)\\ \\
&&+\left(\int_0^{t_2} k_{t_2}^{\Delta_2}(\sigma_1(t_1), t_2, t_1, s_2) q(t_1, s_2)\Delta_2 s_2\right)z(t_1, t_2)\\ \\
&&+\left(\int_0^{t_1} k_{t_1}^{\Delta_1}(t_1, \sigma_2(t_2), s_1, t_2)q(s_1, t_2) \Delta_1 s_1\right)z(t_1, t_2)\\ \\
&&+\left(\int_0^{t_1}\int_0^{t_2} k_{t_1t_2}^{\Delta_1\Delta_2} (t_1, t_2, s_1, s_2) q(s_1, s_2) \Delta_2 s_2\Delta_1 s_1\right)z(t_1, t_2)\\ \\
&=& a(t_1, t_2)+b(t_1, t_2) z(t_1, t_2),\quad  (t_1, t_2)\in (\mathbb{R}_+\bigcap \mathbb{T}_1)\times (\mathbb{R}_+\bigcap \mathbb{T}_2).
\end{eqnarray*}
Hence, using that
\begin{eqnarray*}
z_{t_1}^{\Delta_1}(t_1, 0)&=& 0,\\ \\
z(0, t_2)&=& 0,\quad (t_1, t_2)\in (\mathbb{R}_+\bigcap \mathbb{T}_1)\times (\mathbb{R}_+\bigcap \mathbb{T}_2),
\end{eqnarray*}
we obtain
\begin{eqnarray*}
z_{t_1}^{\Delta_1}(t_1, t_2)-z_{t_1}^{\Delta_1}(t_1, 0)&\leq& \int_0^{t_2} a(t_1, s_2) \Delta_2 s_2+\int_0^{t_2} b(t_1, s_2) z(t_1, s_2)\Delta_2 s_2,\\ \\
z_{t_1}^{\Delta_1}(t_1, t_2)&\leq& \int_0^{t_2} a(t_1, s_2) \Delta_2 s_2+\int_0^{t_2} b(t_1, s_2) z(t_1, s_2)\Delta_2 s_2,
\end{eqnarray*}
\begin{eqnarray*}
z(t_1, t_2)-z(0, t_2) &\leq& \int_0^{t_1} \int_0^{t_2} a(s_1, s_2) \Delta_2s_2 \Delta_1 s_1\\ \\
&&+\int_0^{t_1}\int_0^{t_2} b(s_1, s_2) z(s_1, s_2) \Delta_2 s_2\Delta_1 s_1,\\ \\
z(t_1, t_2)&\leq& A(t_1, t_2)+\int_0^{t_1}\int_0^{t_2} b(s_1, s_2) z(s_1, s_2)\Delta_2 s_2\Delta_1 s_1,
\end{eqnarray*}
$(t_1, t_2)\in (\mathbb{R}_+\bigcap \mathbb{T}_1)\times (\mathbb{R}_+\bigcap \mathbb{T}_2)$.  Now we apply Theorem \ref{theorem6.3} and we get
\begin{equation*}
z(t_1, t_2)\leq A(t_1, t_2) e_c(t_1, 0),\quad (t_1, t_2)\in (\mathbb{R}_+\bigcap \mathbb{T}_1)\times (\mathbb{R}_+\bigcap \mathbb{T}_2),
\end{equation*}
whereupon
\begin{eqnarray*}
u(t_1, t_2)&\leq& p(t_1, t_2)+q(t_1, t_2) z(t_1, t_2)\\ \\
&\leq& p(t_1, t_2) +q(t_1, t_2) A(t_1, t_2) e_c(t_2, 0),
\end{eqnarray*}
$(t_1, t_2)\in (\mathbb{R}_+\bigcap \mathbb{T}_1)\times (\mathbb{R}_+\bigcap \mathbb{T}_2)$. This completes the proof.
\end{proof}
\begin{corollary}
\label{corollary6.5} Let $u, p, q, k\in \mathcal{C}((\mathbb{R}_+\bigcap \mathbb{T}_1)\times (\mathbb{R}_+\bigcap \mathbb{T}_2))$ be nonnegative functions. If
\begin{equation*}
u(t_1, t_2) \leq p(t_1, t_2)+q(t_1, t_2)\int_0^{t_1}\int_0^{t_2} k(s_1, s_2) u(s_1, s_2) \Delta_2 s_2 \Delta_1 s_1,
\end{equation*}
$(t_1, t_2)\in (\mathbb{R}_+\bigcap \mathbb{T}_1)\times (\mathbb{R}_+\bigcap \mathbb{T}_2)$, then
\begin{equation*}
u(t_1, t_2)\leq p(t_1, t_2)+q(t_1, t_2) A(t_1, t_2)e_c(t_1, 0), \quad (t_1, t_2)\in (\mathbb{R}_+\bigcap \mathbb{T}_1)\times (\mathbb{R}_+\bigcap \mathbb{T}_2),
\end{equation*}
where
\begin{eqnarray*}
a(t_1, t_2)&=& k(t_1, t_2)p(t_1, t_2),\\ \\
b(t_1, t_2)&=& k(t_1, t_2) q(t_1, t_2),\\ \\
A(t_1, t_2)&=& \int_0^{t_1}\int_0^{t_2} a(s_1, s_2) \Delta_2 s_2\Delta_1 s_1,\\ \\
c(t_1, t_2)&=& \int_0^{t_2} b(t_1, s_2)\Delta_2 s_2,\quad (t_1, t_2)\in (\mathbb{R}_+\bigcap \mathbb{T}_1)\times (\mathbb{R}_+\bigcap \mathbb{T}_2).
\end{eqnarray*}
\end{corollary}
\begin{theorem}
\label{theorem6.6} Let $c_1$ and $c_2$ be nonnegative constants, $u, v, h_i\in \mathcal{C}( (\mathbb{R}_+\bigcap \mathbb{T}_1)\times (\mathbb{R}_+\bigcap \mathbb{T}_2))$, $i\in\{1, 2, 3, 4\}$, be nonnegative functions. If
\begin{eqnarray*}
u(t_1, t_2)&\leq& c_1+\int_0^{t_1}\int_0^{t_2} h_1(s_1, s_2)u(s_1, s_2)\Delta_2s_2\Delta_1 s_1\\ \\
&&+\int_0^{t_1}\int_0^{t_2} h_2(s_1, s_2) v(s_1, s_2)\Delta_2s_2\Delta_1 s_1,\\ \\
v(t_1, t_2)&\leq& c_2+\int_0^{t_1}\int_0^{t_2} h_3(s_1, s_2) u(s_1, s_2)\Delta_2s_2\Delta_1 s_1\\ \\
&&+\int_0^{t_1}\int_0^{t_2} h_4(s_1, s_2)v(s_1, s_2)\Delta_2s_2\Delta_1 s_1,
\end{eqnarray*}
$(t_1, t_2)\in (\mathbb{R}_+\bigcap \mathbb{T}_1)\times (\mathbb{R}_+\bigcap \mathbb{T}_2)$, then
\begin{equation*}
u(t_1, t_2)+v(t_1, t_2)\leq c_3+A(t_1, t_2) e_c(t_1, 0),\quad (t_1, t_2)\in (\mathbb{R}_+\bigcap \mathbb{T}_1)\times (\mathbb{R}_+\bigcap \mathbb{T}_2),
\end{equation*}
where
\begin{eqnarray*}
c_3&=& c_1+c_2,\\ \\
H(t_1, t_2)&=& \max\{h_1(t_1, t_2)+h_3(t_1, t_2), h_2(t_1, t_2)+h_4(t_1, t_2)\},\\ \\
A(t_1, t_2)&=& c_3\int_0^{t_1}\int_0^{t_2} H(s_1, s_2) \Delta_2 s_2\Delta_1 s_1,\\ \\
c(t_1, t_2)&=& \int_0^{t_2} H(t_1, s_2)\Delta_2s_2,\quad (t_1, t_2)\in (\mathbb{R}_+\bigcap \mathbb{T}_1)\times (\mathbb{R}_+\bigcap \mathbb{T}_2).
\end{eqnarray*}
\end{theorem}
\begin{proof}
Let
\begin{equation*}
f(t_1, t_2)=u(t_1, t_2)+v(t_1, t_2),\quad (t_1, t_2)\in (\mathbb{R}_+\bigcap \mathbb{T}_1)\times (\mathbb{R}_+\bigcap \mathbb{T}_2).
\end{equation*}
Then
\begin{eqnarray*}
f(t_1, t_2)&=& u(t_1, t_2)+v(t_1, t_2)\\ \\
&\leq&  c_1+\int_0^{t_1}\int_0^{t_2} h_1(s_1, s_2)u(s_1, s_2)\Delta_2s_2\Delta_1 s_1\\ \\
&&+\int_0^{t_1}\int_0^{t_2} h_2(s_1, s_2) v(s_1, s_2)\Delta_2s_2\Delta_1 s_1\\ \\
&&+ c_2+\int_0^{t_1}\int_0^{t_2} h_3(s_1, s_2) u(s_1, s_2)\Delta_2s_2\Delta_1 s_1\\ \\
&&+\int_0^{t_1}\int_0^{t_2} h_4(s_1, s_2)v(s_1, s_2)\Delta_2s_2\Delta_1 s_1
\end{eqnarray*}
\begin{eqnarray*}
&=& c_3+\int_0^{t_1}\int_0^{t_2} \left(h_1(s_1, s_2)+h_3(s_1, s_2)\right) u(s_1, s_2)\Delta_2 s_2\Delta_1 s_1\\ \\
&&+\int_0^{t_1}\int_0^{t_2}\left(h_2(s_1, s_2)+h_4(s_1, s_2)\right)v(s_1, s_2)\Delta_2 s_2 \Delta_1 s_1\\ \\
&\leq& c_3+\int_0^{t_1}\int_0^{t_2} H(s_1, s_2) u(s_1, s_2)\Delta_2s_2\Delta_1 s_1\\ \\
&&+\int_0^{t_1}\int_0^{t_2} H(s_1, s_2) v(s_1, s_2) \Delta_2 s_2\Delta_1 s_1\\ \\
&=& c_3+\int_0^{t_1}\int_0^{t_2}H(s_1, s_2)\left(u(s_1, s_2)+v(s_1, s_2)\right)\Delta_2 s_2\Delta_1 s_1\\ \\
&=& c_3+\int_0^{t_1}\int_0^{t_2} H(s_1, s_2) f(s_1, s_2)\Delta_2 s_2\Delta_1 s_1,
\end{eqnarray*}
$(t_1, t_2)\in (\mathbb{R}_+\bigcap \mathbb{T}_1)\times (\mathbb{R}_+\bigcap \mathbb{T}_2)$.  Hence and Corollary \ref{corollary6.5}, we obtain
\begin{eqnarray*}
u(t_1, t_2)+v(t_1, t_2)&=& f(t_1,t_2)\\ \\
&\leq& c_3+A(t_1, t_2) e_c(t_1, 0),\quad (t_1, t_2)\in (\mathbb{R}_+\bigcap \mathbb{T}_1)\times (\mathbb{R}_+\bigcap \mathbb{T}_2).
\end{eqnarray*}
This completes the proof.
\end{proof}
\begin{theorem}[Pachpatte's Inequality]
\label{theorem6.7} Let $u(t_1, t_2)\in \mathcal{C}^2((\mathbb{R}_+\bigcap \mathbb{T}_1)\times (\mathbb{R}_+\bigcap \mathbb{T}_2))$, $u_{t_1, t_2}^{\Delta_1\Delta_2}(t_1, t_2)$ be a nonnegative function and $c(t_1, t_2)$ be a nonnegative continuous function for $(t_1, t_2)\in (\mathbb{R}_+\bigcap \mathbb{T}_1)\times (\mathbb{R}_+\bigcap \mathbb{T}_2)$, and
\begin{equation*}
u(0, t_2)=u(t_1, 0)=0\quad for\quad (t_1, t_2)\in (\mathbb{R}_+\bigcap \mathbb{T}_1)\times (\mathbb{R}_+\bigcap \mathbb{T}_2).
\end{equation*}
Let also, $a\in \mathcal{C}^1(\mathbb{R}_+\bigcap \mathbb{T}_1)$, $b\in \mathcal{C}^1(\mathbb{R}_+\bigcap \mathbb{T}_2)$ be positive functions having derivatives such that
\begin{equation*}
a^{\Delta_1}(t_1)\geq 0,\quad t_1\in \mathbb{R}_+\bigcap \mathbb{T}_1,\quad b^{\Delta_2}(t_2)\geq 0,\quad t_2\in \mathbb{R}_+\bigcap \mathbb{T}_2.
\end{equation*}
If
\begin{equation*}
u_{t_1t_2}^{\Delta_1\Delta_2}(t_1, t_2)\leq a(t_1)+b(t_2) +\int_0^{t_1}\int_0^{t_2} c(s_1, s_2) \left(u(s_1, s_2)+u_{t_1t_2}^{\Delta_1\Delta_2}(s_1, s_2)\right)\Delta_2 s_2\Delta_1 s_1,
\end{equation*}
$(t_1, t_2)\in (\mathbb{R}_+\bigcap \mathbb{T}_1)\times (\mathbb{R}_+\bigcap \mathbb{T}_2)$,  then
\begin{equation*}
u(t_1, t_2)\leq \int_0^{t_1}\int_0^{t_2} h(s_1, s_2)\Delta_2s_2\Delta_1 s_1,\quad (t_1, t_2)\in (\mathbb{R}_+\bigcap \mathbb{T}_1)\times (\mathbb{R}_+\bigcap \mathbb{T}_2),
\end{equation*}
where
\begin{eqnarray*}
p(t_1, t_2)&=& \frac{a^{\Delta_1}(t_1)}{a(t_1)+b(0)}+\int_0^{t_2}\left(1+c(t_1, s_2)\right)\Delta_2 s_2,\\ \\
q(t_1, t_2)&=& \left(a(0)+b(t_2)\right)e_p(t_1, 0) c(t_1, t_2),\\ \\
h(t_1, t_2)&=& a(t_1)+b(t_2) +\int_0^{t_1}\int_0^{t_2} q(s_1, s_2)\Delta_2 s_2\Delta_1 s_1,
\end{eqnarray*}
$(t_1, t_2)\in (\mathbb{R}_+\bigcap \mathbb{T}_1)\times (\mathbb{R}_+\bigcap \mathbb{T}_2)$.
\end{theorem}\vbox{\index{Pachpatte's Inequality}}
\begin{proof}
Let
\begin{equation*}
z(t_1, t_2)= a(t_1)+b(t_2) +\int_0^{t_1}\int_0^{t_2} c(s_1, s_2) \left(u(s_1, s_2)+u_{t_1t_2}^{\Delta_1\Delta_2}(s_1, s_2)\right)\Delta_2 s_2\Delta_1 s_1,
\end{equation*}
$(t_1, t_2)\in (\mathbb{R}_+\bigcap \mathbb{T}_1)\times (\mathbb{R}_+\bigcap \mathbb{T}_2)$. Then
\begin{equation}
\label{6.8} u_{t_1t_2}^{\Delta_1\Delta_2} (t_1, t_2)\leq z(t_1, t_2),\quad (t_1, t_2)\in (\mathbb{R}_+\bigcap \mathbb{T}_1)\times (\mathbb{R}_+\bigcap \mathbb{T}_2),
\end{equation}
\begin{eqnarray*}
z(t_1, 0)&=& a(t_1)+b(0),\quad t_1\in \mathbb{R}_+\bigcap \mathbb{T}_1,\\ \\
z(0, t_2)&=& a(0)+b(t_2),\quad t_2\in \mathbb{R}_+\bigcap \mathbb{T}_2,\\ \\
z_{t_1}^{\Delta_1}(t_1, t_2)&=& a^{\Delta_1}(t_1)+\int_0^{t_2} c(t_1, s_2) \left(u(t_1, s_2)+u_{t_1t_2}^{\Delta_1\Delta_2}(t_1, s_2)\right)\Delta_2 s_2,
\end{eqnarray*}
and
\begin{equation}
\label{6.9} z_{t_1t_2}^{\Delta_1\Delta_2}(t_1, t_2)=c(t_1, t_2)\left(u(t_1, t_2)+u_{t_1t_2}^{\Delta_1\Delta_2}(t_1, t_2)\right),\quad (t_1, t_2)\in (\mathbb{R}_+\bigcap \mathbb{T}_1)\times (\mathbb{R}_+\bigcap \mathbb{T}_2).
\end{equation}
Since
\begin{equation*}
u(0, t_2)=u(t_1, 0)=0,\quad (t_1, t_2)\in (\mathbb{R}_+\bigcap \mathbb{T}_1)\times (\mathbb{R}_+\bigcap \mathbb{T}_2),
\end{equation*}
we have
\begin{equation*}
u_{t_2}^{\Delta_2}(0, t_2)= u_{t_1}^{\Delta_1}(t_1, 0)=0,\quad (t_1, t_2)\in (\mathbb{R}_+\bigcap \mathbb{T}_1)\times (\mathbb{R}_+\bigcap \mathbb{T}_2).
\end{equation*}
Hence and \eqref{6.8}, we obtain
\begin{equation*}
u_{t_1}^{\Delta_1}(t_1, t_2)-u_{t_1}^{\Delta_1}(t_1, 0)\leq \int_0^{t_2} z(t_1, s_2)\Delta_2 s_2,\quad (t_1, t_2)\in (\mathbb{R}_+\bigcap \mathbb{T}_1)\times (\mathbb{R}_+\bigcap \mathbb{T}_2),
\end{equation*}
or
\begin{equation*}
u_{t_1}^{\Delta_1}(t_1, t_2) \leq \int_0^{t_2} z(t_1, s_2)\Delta_2 s_2,\quad (t_1, t_2)\in (\mathbb{R}_+\bigcap \mathbb{T}_1)\times (\mathbb{R}_+\bigcap \mathbb{T}_2),
\end{equation*}
whereupon
\begin{equation*}
u(t_1, t_2)-u(0, t_2)\leq \int_0^{t_1}\int_0^{t_2} z(s_1, s_2)\Delta_2s_2\Delta_1 s_1,\quad (t_1, t_2)\in (\mathbb{R}_+\bigcap \mathbb{T}_1)\times (\mathbb{R}_+\bigcap \mathbb{T}_2),
\end{equation*}
or
\begin{equation*}
u(t_1, t_2)\leq \int_0^{t_1}\int_0^{t_2} z(s_1, s_2)\Delta_2s_2\Delta_1 s_1,\quad (t_1, t_2)\in (\mathbb{R}_+\bigcap \mathbb{T}_1)\times (\mathbb{R}_+\bigcap \mathbb{T}_2).
\end{equation*}
Using the last inequality and the inequality \eqref{6.9}, we get
\begin{equation*}
z_{t_1t_2}^{\Delta_1\Delta_2}(t_1, t_2)\leq c(t_1, t_2) \left(z(t_1, t_2)+\int_0^{t_1}\int_0^{t_2} z(s_1, s_2) \Delta_2 s_2\Delta_1s_1\right),
\end{equation*}
$(t_1, t_2)\in (\mathbb{R}_+\bigcap \mathbb{T}_1)\times (\mathbb{R}_+\bigcap \mathbb{T}_2)$. Let
\begin{equation*}
v(t_1, t_2)= z(t_1, t_2) +\int_0^{t_1}\int_0^{t_2} z(s_1, s_2)\Delta_2 s_2\Delta_1 s_1,\quad (t_1, t_2)\in (\mathbb{R}_+\bigcap \mathbb{T}_1)\times (\mathbb{R}_+\bigcap \mathbb{T}_2).
\end{equation*}
Then
\begin{eqnarray*}
v(t_1, 0)&=& z(t_1, 0),\\ \\
v(0, t_2)&=& z(0, t_2),\quad (t_1, t_2)\in (\mathbb{R}_+\bigcap \mathbb{T}_1)\times (\mathbb{R}_+\bigcap \mathbb{T}_2),
\end{eqnarray*}
and
\begin{eqnarray*}
z(t_1, t_2)&\leq& v(t_1, t_2),\\ \\
z_{t_1t_2}^{\Delta_1\Delta_2}(t_1, t_2)&\leq& c(t_1, t_2) v(t_1, t_2),\quad (t_1, t_2)\in (\mathbb{R}_+\bigcap \mathbb{T}_1)\times (\mathbb{R}_+\bigcap \mathbb{T}_2).
\end{eqnarray*}
Next,
\begin{eqnarray*}
v_{t_1}^{\Delta_1}(t_1, t_2)&=& z_{t_1}^{\Delta_1}(t_1, t_2)+\int_0^{t_2} z(t_1, s_2)\Delta_2 s_2,\\ \\
v_{t_1t_2}^{\Delta_1\Delta_2}(t_1, t_2)&=& z_{t_1t_2}^{\Delta_1\Delta_2}(t_1, t_2)+z(t_1, t_2)\\ \\
&\leq& c(t_1, t_2) v(t_1, t_2)+v(t_1, t_2)\\ \\
&=& \left(1+c(t_1, t_2)\right)v(t_1, t_2),\quad (t_1, t_2)\in (\mathbb{R}_+\bigcap \mathbb{T}_1)\times (\mathbb{R}_+\bigcap \mathbb{T}_2).
\end{eqnarray*}
From here,
\begin{equation*}
\frac{v_{t_1t_2}^{\Delta_1\Delta_2}(t_1, t_2)}{v(t_1, t_2)}\leq 1+c(t_1, t_2),\quad (t_1, t_2)\in (\mathbb{R}_+\bigcap \mathbb{T}_1)\times (\mathbb{R}_+\bigcap \mathbb{T}_2),
\end{equation*}
or
\begin{equation}
\label{6.10} \frac{v_{t_1t_2}^{\Delta_1\Delta_2}(t_1, t_2) v(t_1, t_2)}{\left(v(t_1, t_2)\right)^2}\leq 1+c(t_1, t_2),
\end{equation}
$(t_1, t_2)\in (\mathbb{R}_+\bigcap \mathbb{T}_1)\times (\mathbb{R}_+\bigcap \mathbb{T}_2)$.
Since
\begin{eqnarray*}
z_{t_2}^{\Delta_2}(t_1, t_2)&=& b^{\Delta_2}(t_2) +\int_0^{t_1} c(s_1, t_2) \left(u(s_1, t_2)+u_{t_1t_2}^{\Delta_1\Delta_2}(s_1, t_2)\right)\Delta_1 s_1,\\ \\
&& (t_1, t_2)\in (\mathbb{R}_+\bigcap \mathbb{T}_1)\times (\mathbb{R}_+\bigcap \mathbb{T}_2),
\end{eqnarray*}
and because
\begin{eqnarray*}
b^{\Delta_2}(t_2)&\geq& 0,\quad t_2\in \mathbb{R}_+\bigcap \mathbb{T}_2,\\ \\
c(t_1, t_2)&\geq& 0,\quad u(t_1, t_2)\geq 0,\quad u_{t_1t_2}^{\Delta_1\Delta_2}(t_1, t_2)\geq 0,
\end{eqnarray*}
$(t_1, t_2)\in (\mathbb{R}_+\bigcap \mathbb{T}_1)\times (\mathbb{R}_+\bigcap \mathbb{T}_2)$, we conclude that $z(t_1, t_2)$ is a nondecreasing function with respect to $t_2$. Therefore $v(t_1, t_2)$ is a nondecreasing function with respect to $t_2$. From here
\begin{equation*}
v(t_1, t_2)\leq v(t_1, \sigma_2(t_2)),\quad (t_1, t_2)\in (\mathbb{R}_+\bigcap \mathbb{T}_1)\times (\mathbb{R}_+\bigcap \mathbb{T}_2).
\end{equation*}
Hence, using \eqref{6.10}, we obtain
\begin{equation}
\label{6.11} \frac{v_{t_1t_2}^{\Delta_1\Delta_2}(t_1, t_2)v(t_1, t_2)}{v(t_1, t_2) v(t_1, \sigma_2(t_2))}\leq 1+c(t_1, t_2),\quad (t_1, t_2)\in (\mathbb{R}_+\bigcap \mathbb{T}_1)\times (\mathbb{R}_+\bigcap \mathbb{T}_2).
\end{equation}
Observe that
\begin{eqnarray*}
v_{t_1}^{\Delta_1}(t_1, t_2)&\geq& 0,\\ \\
v_{t_2}^{\Delta_2}(t_1, t_2)&=& z_{t_2}^{\Delta_2}(t_1, t_2)+\int_0^{t_1}z(s_1, t_2) \Delta_1 s_1\\ \\
&\geq& 0, \quad (t_1, t_2)\in (\mathbb{R}_+\bigcap \mathbb{T}_1)\times (\mathbb{R}_+\bigcap \mathbb{T}_2).
\end{eqnarray*}
From here and from \eqref{6.11}, we go to
\begin{eqnarray*}
 \frac{v_{t_1t_2}^{\Delta_1\Delta_2}(t_1, t_2)v(t_1, t_2)}{v(t_1, t_2) v(t_1, \sigma_2(t_2))}&\leq& 1+c(t_1, t_2)\\ \\
 &\leq& 1+c(t_1, t_2) +\frac{v_{t_1}^{\Delta_1}(t_1, t_2) v_{t_2}^{\Delta_2}(t_1, t_2)}{v(t_1, t_2) v(t_1, \sigma_2(t_2))},
\end{eqnarray*}
$(t_1, t_2)\in (\mathbb{R}_+\bigcap \mathbb{T}_1)\times (\mathbb{R}_+\bigcap \mathbb{T}_2)$,
or
\begin{equation*}
 \frac{v_{t_1t_2}^{\Delta_1\Delta_2}(t_1, t_2)v(t_1, t_2)}{v(t_1, t_2) v(t_1, \sigma_2(t_2))}-\frac{v_{t_1}^{\Delta_1}(t_1, t_2) v_{t_2}^{\Delta_2}(t_1, t_2)}{v(t_1, t_2) v(t_1, \sigma_2(t_2))}\leq 1+c(t_1, t_2),
\end{equation*}
$(t_1, t_2)\in (\mathbb{R}_+\bigcap \mathbb{T}_1)\times (\mathbb{R}_+\bigcap \mathbb{T}_2)$,
or
\begin{equation*}
\left(\frac{v_{t_1}^{\Delta_1}(t_1, t_2)}{v(t_1, t_2)}\right)_{t_2}^{\Delta_2}\leq 1+c(t_1, t_2),\quad (t_1, t_2)\in (\mathbb{R}_+\bigcap \mathbb{T}_1)\times (\mathbb{R}_+\bigcap \mathbb{T}_2).
\end{equation*}
From here,
\begin{equation*}
\frac{v_{t_1}^{\Delta_1}(t_1, t_2)}{v(t_1, t_2)}-\frac{v_{t_1}^{\Delta_1}(t_1, 0)}{v(t_1, 0)}\leq \int_0^{t_2} \left(1+c(t_1, s_2)\right)\Delta_2 s_2,\quad (t_1, t_2)\in (\mathbb{R}_+\bigcap \mathbb{T}_1)\times (\mathbb{R}_+\bigcap \mathbb{T}_2),
\end{equation*}
or
\begin{equation*}
\frac{v_{t_1}^{\Delta_1}(t_1, t_2)}{v(t_1, t_2)}-\frac{z_{t_1}^{\Delta_1}(t_1, 0)}{z(t_1, 0)}\leq \int_0^{t_2} \left(1+c(t_1, s_2)\right)\Delta_2 s_2,\quad (t_1, t_2)\in (\mathbb{R}_+\bigcap \mathbb{T}_1)\times (\mathbb{R}_+\bigcap \mathbb{T}_2),
\end{equation*}
or
\begin{equation*}
\frac{v_{t_1}^{\Delta_1}(t_1, t_2)}{v(t_1, t_2)}-\frac{a^{\Delta_1}(t_1)}{a(t_1)+b(0)}\leq \int_0^{t_2} \left(1+c(t_1, s_2)\right)\Delta_2 s_2,\quad (t_1, t_2)\in (\mathbb{R}_+\bigcap \mathbb{T}_1)\times (\mathbb{R}_+\bigcap \mathbb{T}_2),
\end{equation*}
or
\begin{eqnarray*}
\frac{v_{t_1}^{\Delta_1}(t_1, t_2)}{v(t_1, t_2)}&\leq& \frac{a^{\Delta_1}(t_1)}{a(t_1)+b(0)}+ \int_0^{t_2} \left(1+c(t_1, s_2)\right)\Delta_2 s_2\\ \\
&=& p(t_1, t_2),\quad (t_1, t_2)\in (\mathbb{R}_+\bigcap \mathbb{T}_1)\times (\mathbb{R}_+\bigcap \mathbb{T}_2),
\end{eqnarray*}
or
\begin{equation*}
v_{t_1}^{\Delta_1}(t_1, t_2)\leq p(t_1, t_2) v(t_1, t_2),\quad (t_1, t_2)\in (\mathbb{R}_+\bigcap \mathbb{T}_1)\times (\mathbb{R}_+\bigcap \mathbb{T}_2).
\end{equation*}
From the last inequality and from Lemma \ref{lemma2.1}, we obtain
\begin{eqnarray*}
v(t_1, t_2)&\leq& v(0, t_2) e_p(t_1, 0)\\ \\
&=& z(0, t_2) e_p(t_1, 0)\\ \\
&=& \left(a(0)+b(t_2)\right) e_p(t_1, 0),\quad (t_1, t_2)\in (\mathbb{R}_+\bigcap \mathbb{T}_1)\times (\mathbb{R}_+\bigcap \mathbb{T}_2).
\end{eqnarray*}
Therefore
\begin{eqnarray*}
z_{t_1t_2}^{\Delta_1\Delta_2}(t_1, t_2)&\leq&  c(t_1, t_2) v(t_1, t_2)\\ \\
&\leq& \left(a(0)+b(t_2)\right)e_p(t_1, 0) c(t_1, t_2)\\ \\
&=& q(t_1, t_2),\quad (t_1, t_2)\in (\mathbb{R}_+\bigcap \mathbb{T}_1)\times (\mathbb{R}_+\bigcap \mathbb{T}_2),\\ \\
z_{t_1}^{\Delta_1}(t_1, t_2)-z_{t_1}^{\Delta_1}(t_1, 0)&\leq& \int_0^{t_2} q(t_1, s_2)\Delta_2 s_2,\\ \\
z_{t_1}^{\Delta_1}(t_1, t_2)&\leq& a^{\Delta_1}(t_1)+\int_0^{t_2} q(t_1, s_2)\Delta_2 s_2,
\end{eqnarray*}
\begin{eqnarray*}
z(t_1, t_2)-z(0, t_2)&\leq& a(t_1)-a(0)+\int_0^{t_1}\int_0^{t_2} q(s_1, s_2)\Delta_2 s_2\Delta_1 s_1,\\ \\
z(t_1, t_2)&\leq& a(0)+b(t_2)+a(t_1)-a(0)+\int_0^{t_1}\int_0^{t_2} q(s_1, s_2)\Delta_2 s_2\Delta_1 s_1\\ \\
&=& a(t_1)+b(t_2)+\int_0^{t_1}\int_0^{t_2} q(s_1, s_2)\Delta_2 s_2\Delta_1 s_1\\ \\
&=& h(t_1, t_2),\\ \\
u_{t_1t_2}^{\Delta_1\Delta_2}(t_1, t_2)&\leq& h(t_1, t_2),\\ \\
u_{t_1}^{\Delta_1}(t_1, t_2)-u_{t_1}^{\Delta_1}(t_1, 0)&\leq& \int_0^{t_2} h(t_1, s_2)\Delta_2 s_2,\\ \\
u_{t_1}^{\Delta_1}(t_1, t_2)&\leq& \int_0^{t_2} h(t_1, s_2)\Delta_2 s_2,\\ \\
u(t_1, t_2)-u(0, t_2)&\leq& \int_0^{t_1}\int_0^{t_2} h(s_1, s_2)\Delta_2 s_2\Delta_1 s_1,\\ \\
u(t_1, t_2)&\leq& \int_0^{t_1}\int_0^{t_2} h(s_1, s_2) \Delta_2 s_2\Delta_1 s_1,\quad (t_1, t_2)\in (\mathbb{R}_+\bigcap \mathbb{T}_1)\times (\mathbb{R}_+\bigcap \mathbb{T}_2).
\end{eqnarray*}
This completes the proof.
\end{proof}

\end{document}